\documentclass{article}
\usepackage{amsrefs}
\usepackage{amssymb}
\usepackage{amsmath}
\usepackage{amsthm}
\usepackage{graphicx}

\newtheorem{theorem}{Theorem}
\newtheorem{proposition}{Proposition}
\newtheorem{lemma}{Lemma} 

\newtheorem{corollary}{Corollary}
\DeclareMathOperator{\rk}{rk}
\DeclareMathOperator{\val}{val}   \DeclareMathOperator{\Div}{Div}
\newtheorem*{subject}{2000 Mathematics Subject Classification}
\newtheorem*{keywords}{Keywords}

\author{Marc Coppens\footnote{KU  Leuven, Technologiecampus Geel, Departement Elektrotechniek (ESAT),
Kleinhoefstraat 4, B-2440 Geel, Belgium; email: marc.coppens@kuleuven.be.
Partially supported by the FWO-grant 1.5.012.13N}}
\title{A metric graph satisfying $w^1_4=1$ that cannot be lifted to a curve satisfying $\dim (W^1_4)=1$. }
\date{}

\begin{document}
\maketitle \noindent

\begin{abstract}
For all integers $g \geq 6$ we prove the existence of a metric graph $G$ with $w^1_4=1$ such that $G$ has Clifford index 2 and there is no tropical modification $G'$ of $G$ such that there exists a finite harmonic morphism of degree 2 from $G'$ to a metric graph of genus 1.
Those examples show that dimension theorems on the space classifying special linear systems for curves do not all of them have immediate translation to the theory of divisors on metric graphs.
\end{abstract}

\begin{subject}
14H51; 14T05; 05C99
\end{subject}

\begin{keywords}
metric graphs, curves, lifting problems, special divisors, Clifford index, dimension theorems
\end{keywords}

\section{Introduction}\label{section1}

Let $K$ be an algebraically closed field complete with respect to some non-trivial non-Archimedean valuation. Let $R$ be the valuation ring of $K$, let $\mathit{m}_R$ be its maximal ideal and $k=R/\mathit{m}_R$ the residu field.
Let $X$ be a smooth  complete curve of genus $g$ defined over $K$.
Associated to a semistable formal model $\mathfrak{X}$ over $R$ of $X$ there exists a so-called skeleton $\Gamma = \Gamma_ {\mathfrak{X}}$ which is a finite metrix subgraph of the Berkovich analytification $X^{an}$ of $X$ together with an augmentation function $a : \Gamma \rightarrow \mathbb{Z}^+$ such that $a(v)=0$ except for at most finitely many points (see e.g. \cite{ref3}).
In case all components of the special fiber $\mathfrak{X}_k$ are rational then this augmentation function is identically zero and we can consider $\Gamma$ as a metric graph.
This is the situation we consider in this paper.

There exists a theory of divisors and linear equivalence on $\Gamma$ very similar to the theory on curves and those theories on $X$ and $\Gamma$ are related by means of a specialisation map
\[
\tau_* : \Div (X) \rightarrow \Div(\Gamma) \text { .}
\]
For a divisor $E$ on $\Gamma$ one defines a rank $\rk(E)$ and for a divisor $D$ on $X$ the specialisation theorem says (see e.g. \cite{refextra}, since we restrict to the case of zero augmentation map this is in principle considered in \cite{ref4})
\[
\dim (\vert D \vert) \leq \rk (\tau_* (D)) \text { .}
\]

In the hyperelliptic case many classical results on linear systems on curves also do hold for linear systems on metric graphs.
As an example, if the graph $\Gamma$ has a very special linear system $g^r_{2r}$ then $\Gamma$ has a $g^1_2$ and $g^r_{2r}=rg^1_2$ (see \cite{ref15}; \cite{ref5}).
Hence the theory of linear systems of Clifford index 0 is the same for graphs as for curves.
This is not true for the theory of linear systems of Clifford index more than 0 (see \cite{ref2}).

For a curve $X$ the complete linear systems $g^n_d$ with $n \geq r$ are parametrized by a closed subscheme $W^r_d$ of the Jacobian $J(X)$ and $d-2r-\dim (W^r_d)$ gives a kind of generalisation of the Clifford index for moving linear systems on $X$.
In particular in case $r\leq g-1$ then $\dim (W^r_d)\leq d-2r$ and $\dim (W^r_d)=d-2r$ for some $r<g-1$ if and only if $X$ is hyperelliptic (see \cite{ref16}).
In \cite{ref13}  it is shown that using the dimension of a similar subspace $W^r_d$ of the Jacobian $J(\Gamma)$ of a metric graph $\Gamma$ this statement is not true.
Moreover in that paper the authors do introduce a much better invariant $w^r_d$ which is more close to the definition of the rank of a graph.
In \cite{ref9} it is proved that $w^r_d \geq \dim(W^r_d)$ in case $\Gamma$ is a skeleton of a curve $X$ over $K$.
At the moment is seems not known whether $w^r_d\geq d-2r$ for some $0<r<g-1$ implies $\Gamma$ is hyperelliptic.

In \cite{ref8} one finds a classification of all curves $X$ such that $\dim (W^r_d)=d-2r-1$.
Many  more generalisations are proved by different authors.
In this paper we show that the theory of curves satisfying $\dim (W^1_4)=1$ is different from the theory of graphs satisfying $w^1_4=1$.
In particular for all genus $g \geq 6$  we prove the existence of a metric graph $G_n$ of genus $g$ ($n=g-3$) satisfying $w^1_4=1$ that cannot be a skeleton of a curve satisfying $\dim (W^1_4)=1$.

In Section \ref{section2} we recall some generalities on graphs and the theory of divisors on graphs.
For generalities on the specialisation map and the relation between the metric graphs and skeleta inside Berkovich curves we refer to the references.
It is not needed to understand the arguments used in this paper, it is important for the motivation.
In Section \ref{section3} we give the description of the graph denoted by $G_n$ ($n$ an integer at least equal to 2) and we prove it satisfies $w^1_4=1$ and it has Clifford index equal to 2.
In Section \ref{section4} first we explain that in case $G_n$ could be lifted to a curve satisfying $\dim (W^1_4)=1$ then for some tropical modification $G'_n$ of $G_n$ there would exist a finite harmonic morphism $\pi : G'_n \rightarrow \Gamma$ of degree 2 with $\Gamma$ a metric graph of genus 1.
Finally in Section \ref{section4} we prove that such harmonic morphism does not exist.

\section{Generalities}\label{section2}

\subsection{Graphs}\label{subsection2.1}

A \emph{topological graph} $\Gamma$ is a compact topological space such that for each $P\in \Gamma$ there exists $n_P \in \mathbb{Z}^+$ and $\epsilon \in \mathbb{R}^+_0$ such that some neighborhood $U_P$ of $P$ in $\Gamma$ is homeomorphic to $\{ z = re^{2\pi ik/n_P} : 0 \leq r \leq \epsilon \text { and $k$ is an integer satisfying } 0 \leq k \leq n_P-1 \} \subset \mathbb{C}$ with $P$ corresponding to 0.
Such a topological graph $\Gamma$ is called \emph{finite} in case there are only finitely many points $P \in \Gamma$ satisfying $n_P \neq 2$.
We only consider finite topological graphs. 
We call $n_P$ the \emph{valence} of $P$ on $\Gamma$.
Those finitely many points of $\Gamma$ are called the \emph{essential vertices} of $\Gamma$.
The \emph{tangent space} $T_P(\Gamma)$ of $\Gamma$ at $P$ is the set of $n_P$ connected components of $U_P \setminus \{ P \}$ for $U_P$ as above.
In this definition, using another such neighborhood $U'_P$ then we identify connected components of $U_P \setminus \{ P \}$ and $U'_P \setminus \{ P \}$ in case their intersection is not empty.

A metric graph $\Gamma$ is a finite topological graph $\Gamma$ together with a finite subset $V_{\infty}(\Gamma)$ of the set of 1-valent points of $\Gamma$ and a complete metric on $\Gamma \setminus V_{\infty}(\Gamma)$.
A \emph{vertex set} of a metric graph $\Gamma$ is a finite subset of $\Gamma$ containing all essential vertices.
The pair $(\Gamma , V)$ is called a \emph{metric graph with vertex set $V$}.
The elements of $V$ are called the vertices of $(\Gamma , V)$.
The connected components of $\Gamma \setminus V$ are called the \emph{edges} of $(\Gamma , V)$.
The elements of $\overline{e} \setminus e$ are called the \emph{end vertices} of $e$ ($\overline{e}$ is the closure of $e$).
We always choose $V$ such that each edge has two different end vertices.
Using the metric on $\Gamma \setminus V_{\infty}(\Gamma)$ each edge $e$ of $\Gamma$ has a lenght $l(e) \in \mathbb{R}^+_0 \cup \{ \infty \}$.
Moreover $l(e)=\infty$ if and only if some end vertex of $e$ belongs to $V_{\infty}(\Gamma)$.
We write $E(\Gamma , V)$ to denote the set of edges of $(\Gamma , V)$.
The \emph{genus} of $(\Gamma, V)$ is defined by $\vert E(\Gamma ) \vert - \vert V(\Gamma )\vert +1$ and it is independent of the choice of $V$. Hence it is denoted by $g(\Gamma )$ and called the genus of $\Gamma$.

A \emph{subgraph} of a metric graph $(\Gamma, V)$ with vertex set is a closed subset $\Gamma' \subset \Gamma$ such that $(\Gamma ', \Gamma ' \cap V)$ is a metric graph with vertex set.
In case $\Gamma'$ is homeomorphic to $S^1$ then it is called a \emph{loop} in $(\Gamma, V)$.
A metric graph $\Gamma$ is called a \emph{tree} if $g(\Gamma )=0$.

\subsection{Linear systems on graphs}\label{subsection2.2}

We refer to Section 2 of \cite{ref2} for the definitions on a metric graph of a divisor, an effective divisor, a rational function, linearly equivalence of divisors, the canonical divisor and the rank of a divisor.
The rank of a divisor $D$ on a metric graph $\Gamma$ is denoted by $\rk (D)$ and if it is necessary to add the graph then we write $\rk_{\Gamma} (D)$.
For a divisor $D$ on a graph $\Gamma$ we write $D \geq 0$ to indicate it is an effective divisor on $\Gamma$.
A very important tool in the study of divisors on a metric graph $\Gamma$ is the concept of a reduced divisor at some point $P$ of $\Gamma$ (see Section 2.1 in \cite{ref2}) and the burning algorithm to decide whether a given divisor on $\Gamma$ is reduced at $P$ (see Section 2.2 in \cite{ref2}).

For a divisor $D$ on a metric graph $\Gamma$ we write $\vert D \vert$ to denote the set of effective divisors linearly equivalent to $D$.
As is the case of curves we call it the \emph{complete linear system} defined by $D$.
The rank $\rk (D)$ replaces the concept of the dimension of a complete linear system on a curve.
As in the case of curves we say the complete linear system $\vert D \vert$ is a linear system $g^r_d$ on $\Gamma$ if $\deg (D)=d$ and $\rk (D)=r$.
From of the Riemann-Roch Theorem for divisors on graphs it follows as in the case of curves that divisors $D$ on $\Gamma$ such that $\rk (D)$ cannot be computed in a trivial way are exactly those divisors satisfying $\rk (D) > \max \{ 0, \deg (D) -g(\Gamma )+1 \}$.
Those divisors are called \emph{very special}.
The \emph{Clifford index} of a very special divisor $D$ on $\Gamma$ is defined by $c(D)=\deg (D) - 2 \rk (D)$.
Clifford's Theorem for metric graphs implies $c(D) \geq 0$ for all very special divisors $D$ on $\Gamma$.
The Clifford index $c(\Gamma)$ of $\Gamma$ is the minimal value $c(D)$ for a very special divisor $D$ on $\Gamma$.

Motivated by the definition of the rank of a divisor on a metric graph one introduces the following replacement for the dimension of the space $W^r_d$ parametrizing linear systems $g^r_d$ on a curve.
In case $\Gamma$ has no linear system $g^r_d$ then $w^r_d=-1$.
Otherwise $w^r_d$ is the maximal integer $w\geq 0$ such that for each effective divisor $F$ of degree $r+w$ there exists an effective divisor $E$ of degree $d$ with $\rk (E)\geq r$ such that $E-F \geq 0$.

\subsection{Harmonic morphism}\label{subsection2.3}

Let $\Gamma$ and $\Gamma'$ be two metric graphs and let $\phi : \Gamma \rightarrow \Gamma'$ be a continuous map.
In case $V$ (resp. $V'$) is a vertex set of $\Gamma$ (resp. $\Gamma'$) then $\phi$ is called a \emph{morphism} from $(\Gamma', V')$ to $(\Gamma , V)$  if $\phi (V') \subset V$ and for each $e \in E(\Gamma , V)$ the set $\phi ^{-1} (\overline{e})$ is a union of closures of edges of $(\Gamma' , V')$.
Moreover if $e' \in E(\Gamma ', V')$ with $e' \subset \phi^{-1}(\overline{e})$ then either $\phi(e')$ is a vertex in $V'$ or the restriction $\phi_{e'} : e' \rightarrow e$ is a dilation with some factor $d_{e'}(\phi) \in \mathbb{Z}^+_0$.
In case $\phi (e')$ is a vertex then we write $d_{e'}(\phi)=0$.
We call $d_{e'}(\phi)$ the \emph{degree} of $\phi$ along $e'$.

We say $\phi$ is a morphism of metric graphs if there exist vertex sets $V$ (resp. $V'$) of $\Gamma$ (resp. $\Gamma'$) such that $\phi$ is a morphism from $(\Gamma' , V')$ to $(\Gamma, V)$.
In that case, for $v' \in T_{P'}(\Gamma')$ and $e'$ an edge of $(\Gamma ', V')$ such that $v'$ is defined by some connected component of $e' \setminus \{P' \}$ we set $d_{v'}(\phi)=d_{e'}(\phi)$.
Such morphism is called \emph{finite} in case $d_{e'}(\phi) >0$ for all $e' \in E(\Gamma ', V')$.
There is a natural map $d_{\phi}(P') : T_{P'} \setminus \{ v' : d_{v'}(\phi)=0 \} \rightarrow T_{\phi (P')}(\Gamma)$ defined as follows.
The connected component of $e' \setminus \{ P' \}$ defining $v' \in T_{P'} (\Gamma')$ with $d_{v'}(\phi) \neq 0$ is mapped to a connected component of $\phi (e') \setminus \{ \phi(P') \}$ and this defines $v \in  T_{\phi (P')}(\Gamma)$, then $d_{\phi}(P')(v')=v$.

The morphism $\phi : \Gamma ' \rightarrow \Gamma$ of metric graphs is called \emph{harmonic} at $P' \in \Gamma'$ if for each $v \in T_{\phi(P'}(\Gamma )$ the number
\[
\Sigma \{d_{v'}(\phi) : v' \in T_{P'}(\Gamma ') \text { and } d_{\phi}(P')(v')=v \}
\]
is independent of $v$.
In that case this sum is denoted by $d_{P'}(\phi)$ and it is called the \emph{degree} of $\phi$ at $P'$.
We say the morphism $\phi$ is harmonic if $\phi$ is surjective and $\phi$ is harmonic at each point $P' \in \Gamma '$.
In this case for $P \in \Gamma$ one has $\Sigma (d_{P'}(\phi) : \phi (P')=P )$ is independent of $P$ and it is called the \emph{degree} of $\phi$ denoted by $\deg (\phi)$.

An \emph{elementary tropical modification} of a metric graph $\Gamma$ is a metric graph $\Gamma '$ obtained by attaching an infinite closed edge to $\Gamma$ at some point $P\in \Gamma \setminus V_{\infty}(\Gamma)$.
A metric graph obtained from $\Gamma$ as a composition of finitely many elementary tropical modifications is called a \emph{tropical modification} of $\Gamma$.
Two metric graphs $\Gamma_1$ and $\Gamma_2$ are called \emph{tropically equivalent} if there is a common tropical modification $\Gamma$ of $\Gamma _1$ and $\Gamma _2$.
This terminology can be found in e.g. \cite{ref17} together with some examples. 

\section{The example}\label{section3}

The metric graph $G_0$ we start with has genus 2 and can be seen in figure \ref{Figuur 1}.
\begin{figure}[h]
\begin{center}
\includegraphics[height=5 cm]{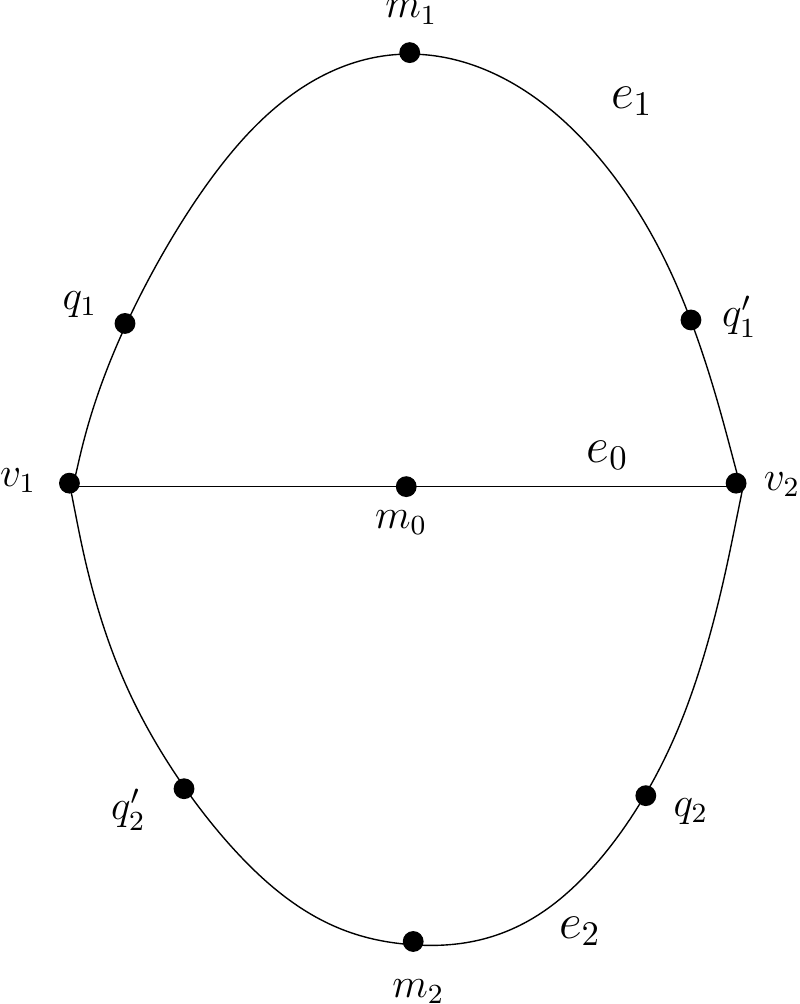}
\caption{the graph $G_0$}\label{Figuur 1}
\end{center}
\end{figure}
Here $v_1$ and $v_2$ are two points of valence 3 (all other points have valence 2) and they are connected by three edges $e_0$, $e_1$ and $e_2$ of mutually different lenghts.
For $0 \leq i \leq 2$ the point $m_i$ is the midpoint of $e_i$.

\begin{lemma}\label{lemma1}
The graph $G_0$ has a unique $g^1_2$ given by $\vert v_1 + v_2 \vert = \vert 2m_0 \vert = \vert 2m_1 \vert = \vert 2m_2 \vert$ and in case $v \in G_0$ such that $2v \in \vert v_1 + v_2 \vert$ then $v=m_i$ for some $0 \leq i \leq2$.
\end{lemma}
\begin{proof}
Clearly $2m_i \in \vert v_1 + v_2 \vert$ for $0 \leq i \leq 2$ and in case $v \in e_i \setminus \{ v_1, v_2, m_i \}$ then taking $v'$ on $e_i$ such that the distance on $e_i$ from $v$ to $v_1$ is equal to the distance of $v'$ to $v_2$, then $v+v' \in \vert v_1+v_2 \vert$.
This proves $\rk (v_1 + v_2)=1$ (it cannot have rank 2 because $g(G_0)\neq 0$).
It is well-known that a graph of genus at least 2 has at most one $g^1_2$ (see \cite{ref1}).
Indeed, for this graph $G_0$, if $v \neq v_2$ then $v_1 + v$ is clearly $v_2$-reduced, hence $\vert v_1 + v - v_2 \vert = \emptyset$ and therefore $\rk (v_1 + v)=0$.
This proves the uniqueness of $g^1_2$ on $G_0$.

Finally for $v+v' \in \vert v_1 + v_2 \vert$ as before (including the possibility $v+v' = v_1 + v_2$), since $v'$ is a $v$-reduced divisor one has $\vert v'-v \vert = \emptyset$ hence $\vert v_1 + v_2 -2v \vert = \emptyset$.
This proves $2v \in \vert v_1 + v_2 \vert$ implies $v=m_i$ for some $0 \leq i \leq2$.
\end{proof}

As indicated in figure \ref{Figuur 1} we fix $q_i \in ] v_i,m_i [ \subset e_i$ for $i=1,2$.

\begin{lemma}\label{lemma2}
There is no $g^1_3$ on $G_0$ such that $\vert g^1_3 - 2m_0 \vert \neq \emptyset$ and $\vert g^1_3 -2q_i \vert \neq \emptyset$ for $i=1,2$.
\end{lemma}

\begin{proof}
For $i=1,2$ we take $q'_i \in e_i$ such that the distance on $e_i$ from $v_1$ to $q_i$ is equal to the distance on $e_i$ from $v_2$ to $q'_i$.
Assume $g^1_3$ on $G_0$ with $\vert g^1_3 - 2m_0 \vert \neq \emptyset$, hence there exists $v \in G_0$ such that $g^1_3 = \vert g^1_2 + v \vert$.

First assume $v \in (e_0 \cup e_2) \setminus \{v_1 \}$.
Then $q_1 + q'_1 + v \in g^1_3$ and clearly $q'_1 + v$ is a $q_1$-reduced divisor.
This implies $\vert q'_1+v-q_1 \vert = \vert g^1_3 - 2q_1 \vert = \emptyset$.
In case $v \notin (e_0 \cup e_2) \setminus \{ v_1 \}$ then certainly $v \in (e_0 \cup e_1) \setminus \{v_2 \} $ and using similar arguments we find $\vert g^1_3 - 2q_2 \vert = \emptyset$.
This finishes the proof of the lemma.
\end{proof}

Now for an integer $n\geq 2$ we make a graph $G_n$ as follows.
Fix some more general different points $q_3, \cdots , q_n$ on $G_0$.
Then $G_n$ is obtained from $G_0$ by attaching a loop $\gamma_0$ at $m_0$ and loops $\gamma_i$ at $q_i$ for each $1 \leq i\leq n$ (we also are going to denote $m_0$ by $q_0$).
As an example see a possible picture of $G_6$ in figure \ref{Figuur 2}.
\begin{figure}[h]
\begin{center}
\includegraphics[height=5 cm]{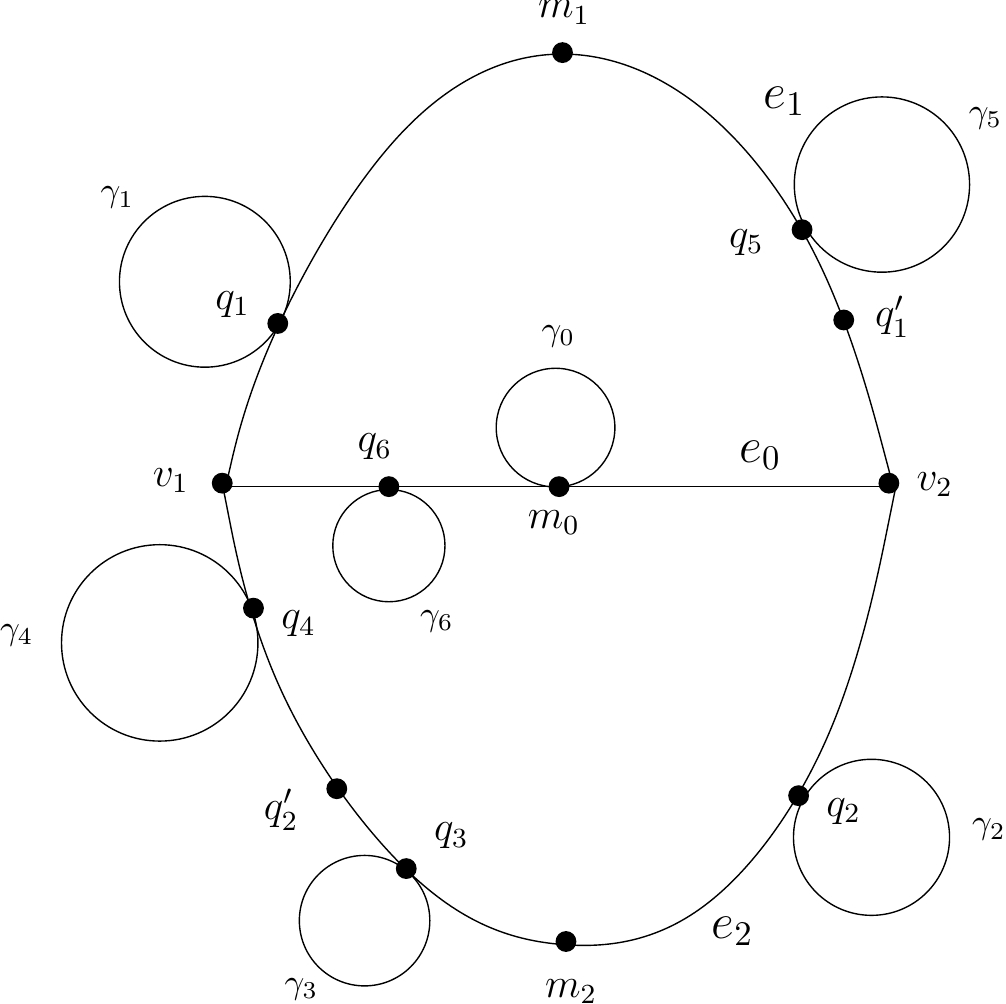}
\caption{the graph $G_6$}\label{Figuur 2}
\end{center}
\end{figure}
Clearly $g(G_n)=n+3$.
We prove that the Clifford index of $G_n$ is at least 2.

\begin{proposition}\label{proposition1}
Let $r$ be an integer with $1 \leq r \leq n$.
Then $G_n$ has no $g^r_{2r+1}$ in case $n\geq 2$.
\end{proposition}

\begin{proof}

First we show $G_n$ has no linear system $g^1_2$ in case $n\geq 1$ (here $G_1=G_0 \cup \gamma_0 \cup \gamma_1$).
Assume there is a $g^1_2$ on $G_n$.
Take $1 \leq i \leq n$ and let $v \in \gamma_i \setminus \{ q_i \}$ and $v' \in G_n$ such that $v+v' \in g^1_2$.
Let $G^0_n$ be the closure of $G_n \setminus \gamma_i$ and assume $v' \in G^0_n$.
It follows from Corollary 1 in \cite{ref2} that $\rk _{G^0_n}(v')\geq 1$, because $g(G^0_n)>0$ this is impossible.
Hence $v' \in \gamma_i$.
On $\gamma_i$ there exists $v"$ such that $v + v'$ is linearly equivalent to $q_i + v"$ as divisors on $\gamma_i$, hence $q_i + v" \in g^1_2$ because of Lemma 1 in \cite{ref2} (that lemma does not depend on the particular graph used in \cite{ref2}).
But we proved this implies $v"=q_i$, hence $2q_i \in g^1_2$. 
This implies $\rk_{G_0}(2q_i)=1$.
Indeed, take $p\in G_0 \setminus \{ q_i \}$ and let $D_p$ be the $p$-reduced divisor on $G_0$ linearly equivalent to $2q_i$.
We need to show $D_p - p \geq 0$.
The burning algorithm applied to $G_n$ implies $D_p$ is a $p$-reduced divisor on $G_n$ too.
Since $\rk_{G_n}(2q_i)=1$ it follows $D_p-p \geq 0$.
So we obtain $\rk _{G_0}(2q_i)=1$, but from Lemma \ref{lemma1} we know this cannot be true, hence $G_n$ has no linear system $g^1_2$.

Fix some integer $r$ satisfying $1\leq r\leq n-1$ and assume $G_n$ has a linear system $g^r_{2r+1}$.
For $0\leq i\leq r+1$ fix $v_i \in \gamma_i$ with $v_i \neq q_i$.
For $0 \leq i \leq 2$ there exists $E_i \in g^r_{2r+1}$ satisfying
\[
E_i - (v_i + v_3 + \cdots v_{r+1}) \geq 0 \text { .}
\]
For $j \in \{ 3, \cdots , r+1 \} \cup \{ i \}$ let $D_{i,j}=E_i \cap (\gamma_j \setminus \{ q_j \})$, hence $D_{i,j} - v_j \geq 0$.
In case for some $j$ the $q_j$-reduced divisor on $\gamma_j$ linearly equivalent to $D_{i,j}$ contains a point $v'_j$ different from $q_j$ (then the point $v'_j$ is unique) then there is an effective divisor $E'$ on $\overline{G_n \setminus \gamma_j}$ of degree $2r$ such that $E'+v'_j \in g^r_{2r+1}$.
From Corollary 1 in \cite{ref2} it follows $\rk _{\overline{G_n \setminus \gamma_j}}(E')=r$, hence $\overline{G_n \setminus \gamma_j}$ has a linear system $g^r_{2r}$.
Since $g(\overline{ G_n \setminus \gamma_j})=n+2$ and $2\leq 2r \leq 2n-2$ this would imply $\overline{G_n \setminus \gamma_j}$ is a hyperelliptic graph (main result of \cite{ref15}).
But we proved $\overline{G_n \setminus \gamma_j}$ is not hyperelliptic (it is a graph $G_{n-1}$; the proof of that argument also works on $\overline{G_n \setminus \gamma _0}$ in case $j=0$).
Since $v_j$ is not linearly equivalent to $q_j$ as a divisor on $\gamma_j$ it follows $D_{i,j}$ is linearly equivalent to $m_{i,j}q_j$ for some $m_{i,j}\geq 2$ on $\gamma_j$.

So we obtain $E'_i \in g^r_{2r+1}$ on $G_n$ such that
\[
E'_i - (2q_i + 2q_3 + \cdots + 2q_{r+1}) \geq 0
\]
for $0 \leq i \leq 2$ and $E'_i$ is contained in $G_0$.
Because of Lemma 2 in \cite{ref2} those divisors $E'_i$ are linearly equivalent as divisors on $G_0$.
It follows $E"_i=E'_i-(2q_3 + \cdots + 2q_{r+1})$ with $0 \leq i \leq 2$ are effective linearly equivalent divisors on $G_0$ of degree 3 with $E"_i - 2q_i \geq 0$.
Since $g(G_0)=2$ each divisor of degree 3 on $G_0$ defines a $g^1_3$ and we obtain a contradiction to Lemma \ref{lemma2}.

Finally if $G_n$ has an $g^n_{2n+1}$ then because of the Theorem of Riemann-Roch $\vert K_{G_n}-g^n_{2n+1} \vert = g^1_3$, we already excluded this case.
\end{proof}

\begin{proposition}\label{proposition2}
On $G_n$ we have $w^1_4=1$.
\end{proposition}
\begin{proof}
Fix $v_1, v_2 \in G_n$.
We need to prove that there exists a $g^1_4$ on $G_n$ such that $\vert g^1_4 -v_1 - v_2 \vert \neq \emptyset$.
In case $D$ is an effective divisor of degree 4 on $G_0$ then from the Riemann-Roch Theorem it follows $\rk _{G_0} (D)=2$.
This implies for each $0 \leq i \leq n$ there is an effective divisor $D' \geq 2q_i$ linearly equivalent to $D$ on $G_0$.
Because of Lemma 2 in \cite{ref2} the divisor $D'$ is linearly equivalent to $D$ on $G_n$.
Moreover for $v \in \gamma_i$ there is an effective divisor on $\gamma_i$ linearly equivalent to $2q_i$ containing $v$ and using the same lemma we obtain the existence of an effective divisor on $G_n$ linearly equivalent to $D$ containing $v$.
Similarly, for $v \in G_0$ we obtain an effective divisor on $G_0$ linearly equivalent to $D$ and containing $v$ and again this divisor is also linearly equivalent to $D$ as a divisor on $G_n$.
This proves $\rk _{G_n}(D)\geq 1$.

In case $v_1, v_2 \in G_0$ we can use an effective divisor $D$ on $G_0$ containing $v_1 + v_2$.
Then we have $\rk_{G_n}(D)\geq 1$ and $\vert D-v_1 -v_2 \vert \neq \emptyset$.
Next assume $v_1 \in G_0$ and $v_2 \in \gamma_i \setminus \{ q_i \}$ for some $0 \leq i \leq n$.
On $G_0$ take an effective divisor $D$ of degree 4 containing $v_1+2q_i$.
Since $2q_i$ is linearly equivalent to $v_2+ v'_2$ for some $v'_2\in \gamma_i$ as a divisor on $\gamma_i$, again using lemma 1 in \cite{ref2} we find that $D-2q_i+v_2+v'_2$ is an effective divisor linearly equivalent to $D$ on $G_n$.
Hence $\rk_{G_n}(D)\geq 1$ and $\vert D-v_1-v_2 \vert \neq \emptyset$.
Assume $v_1 \in \gamma_{i_1} \setminus \{ q_{i_1} \}$ and  $v_2 \in \gamma_{i_2} \setminus \{ q_{i_2} \}$ for some $i_1 \neq i_2$.
Then one makes a similar argument using the divisor $2q_{i_1}+2q_{i_2}$ on $G_0$.
Finally if $v_1, v_2 \in \gamma_i \setminus \{ q_i \}$ for some $i$ then one uses an effective divisor $D$ of degree 4 on $G_0$ containing $3q_i$.
Since $3q_i$ is linearly equivalent on $\gamma_i$ to an effective divisor $D$ containing $v_1+v_2$ again we find $\rk _{G_n}(D)\geq 1$ and $\vert D-v_1-v_2 \vert \neq \emptyset$.
\end{proof}

\section{The lifting problem}\label{section4}

The lifting problem associated to $G_n$ we are going to consider now is the following.
Let $K$ be an algebraically closed complete non-archimedean valued field and let $X$ be a smooth algebraic curve of genus $g$.
Let $X^{an}$ be the analytification of $X$ (as a Berkovich curve).
Let $R$ be the valuation ring of $K$, let $\mathfrak{X}$ be a strongly semistable model of $X$ over $R$ such that the special fiber has only rational components and let  $\Gamma$ be the associated skeleton.
Is it possible to obtain this situation such that $\Gamma = G_n$ and $\dim (W^1_4 (X))=1$?
In that case, taking into account the result from \cite{ref9} mentioned in the introduction, this would give a geometric explanation for $w^1_4(G_n)=1$.
This lifting problem will be the motivation for considering the existence of a certain harmonic morphism associated to $G_n$.
Making this motivation we are going to refer to some suited papers for terminology and some definitions.
The definitions necessary to understand the question on the existence of the harmonic morphism are given in Section \ref{subsection2.3}.
Finally we are going to prove that the harmonic morphism does not exist, proving that the lifting problem has no solution.
In particular we obtain that the classification of metric graphs satisfying $w^1_4 =1$ is different from the classification of smooth curves satisfying $\dim (W^1_4)=1$.

Assume the lifting problem has a solution.
The curve $X$ of that solution cannot be hyperelliptic since $G_n$ is not hyperelliptic.
From \cite{ref8} one obtains the following classification of non-hyperelliptic curves $X$ of genus at least 6 satisfying $\dim (W^1_4(X))=1$: $X$ is trigonal (has a $g^1_3$), $X$ is a smooth plane curve of degree 5 (hence has genus 6 and has a $g^2_5$) or $X$ is bi-elliptic (there exists a double covering $\pi : X \rightarrow E$ with $g(E)=1$).
From Proposition \ref{proposition1} we know $G_n$ has no $g^1_3$ and no $g^2_5$, hence the curve $X$ has to be bi-elliptic.

So assume there exists a morphism $\pi : X \rightarrow E$ with $g(E)=1$ of degree 2.
This induces a map $\pi^{an}:X^{an} \rightarrow E^{an}$ between the Berkovich analytifications.
In case $E$ is not a Tate curve then each strong semistable reduction of $E$ contains a components of genus 1 in its special fiber, in particular the augmentation map of the associated skeleton has a unique point with value 1.
Otherwise such skeleton can be considered as a metric graph of genus 1.
Each skeleton associated to a semistable reduction of $X$ is tropically equivalent to the graph $G_n$, in particular it can be considered as a metric graph.
From the results in \cite{ref12} it follows that there exist skeletons $\tilde{\Gamma}$ (resp. $\Gamma$) of $X$ (resp. $E$) such that $\pi$ induces a finite harmonic morphism $\tilde{\Gamma} \rightarrow \Gamma$ of degree 2.
Since $\tilde{\Gamma}$ is a metric graph (augmentation map identically zero) this is also the case for $\Gamma$ hence $\Gamma$ is a metric graph.
So in case $G_n$ is liftable to smooth curve $X$ satisfying $\dim (W^1_4(X))=1$ then there exist a tropical modification $\tilde{\Gamma}$ of $G_n$ and a metric graph $\Gamma$ of genus 1 such that there exists a finite harmonic morphism $\tilde {\pi} : \tilde {\Gamma} \rightarrow \Gamma$ of degree 2.
We are going to prove that such finite harmonic morphism does not exist.
In the proof the following lemma will be useful.

\begin{lemma}\label{lemma3}
Let $\phi : (\Gamma _1, V_1) \rightarrow (\Gamma_2 , V_2)$ be a finite harmonic morphism between metric graphs with vertex sets.
Let $(T' , V') \subset (\Gamma_1, V_1)$ be a subgraph such that $T'$ is a tree, $\overline{\Gamma_1 \setminus T'} \subset \Gamma_1$ is connected and $\overline{ \Gamma_1 \setminus T'} \cap T'$ consists of a unique point $t$ (in particular $t\in V_1$).
There is no subtree $(T,V)$ of $(T',V')$ different from a point such that $\phi (T)$ is contained in a loop $\Gamma \subset \Gamma_2$.
\end{lemma}

\begin{proof}
Assume $T$ is a subtree of $T'$ not being one point and assume $\phi (T)$ is contained in a loop $\Gamma$ of $\Gamma_2$.
Let $l(T)$ (resp. $l(\Gamma)$) be the sum of the lenghts of all the edges of $T$ (resp. $\Gamma$).
By definition one has $l(T) \leq \deg (\phi) l(\Gamma)$.
We are going to prove that we have to be able to enlarge $T$ such that $l(T)$ grows with a fixed lower bound.
Repeating this a few times gives a contradiction to the upper bound $\deg (\phi) l(\Gamma)$.

Let $q \in V$ be a point of valence 1 on $T$ such that $q \neq t$ and let $f$ be the edge of $T$ having $q$ as a vertex point.
This edge $f$ defines $v \in T_q(\Gamma_1)$, let $w=d_{\phi}(q)(v)$, hence $\phi(q) \in \Gamma$ and $w \in T_{\phi(q)}(\Gamma)$.
Since $\Gamma$ is a loop there is a unique $w' \in T_{\phi (q)}(\Gamma)$ with $w' \neq w$ and since $\phi$ is harmonic there exists $v' \in T_q(\Gamma_1)$ with $d_{\phi}(q)(v')=w'$.
Let $f'$ be the edge in $\Gamma_1$ having $q$ as a vertex point and defining $v'$.
Since $f' \neq f$ and $q \neq t$ one has $f'$ is an edge of $\overline{T' \setminus T}$.
Since $T'$ is a tree, also $T' \cup f'$ is a tree and one has $\phi (T \cup f') \subset \Gamma$.
Moreover $l(T \cup f') = l(T) + l(f')$ and $l(f')$ has as a fixed lower bound the minimal length of an edge contained in $\Gamma_1$.
\end{proof}

\begin{theorem}
There does not exist a tropical modification $G'_n$ of $G_n$ such that there exists a graph $E$ with $g(E)=1$ and a finite harmonic morphism $\phi : G'_n \rightarrow E$ of degree 2.
\end{theorem}

\begin{proof}
Assume $G'_n$ is a tropical modification of $G_n$ and $\phi : G'_n \rightarrow E$ is a finite harmonic morphism of degree 2 of metric graphs with $g(E)=1$.

\noindent \underline{Step 1:} $g(\phi (G_0))=1$.

Assume $g(\phi (G_0))=0$.
We are going to prove that $\phi (G_0)$ looks as in figure \ref{Figuur 3} with $\phi \bigr| _{v_1m_iv_2} : v_1m_iv_2 \rightarrow [\phi (v_1), \phi (m_i)]$ has degree 2 and $\sharp (( \phi \bigr| _{v_1m_iv_2} )^{-1} (q))=2$ for all $q \in [ \phi(v_1), \phi(m_i) [$ for $0 \leq i \leq 2$. 
\begin{figure}[h]
\begin{center}
\includegraphics[height=3 cm]{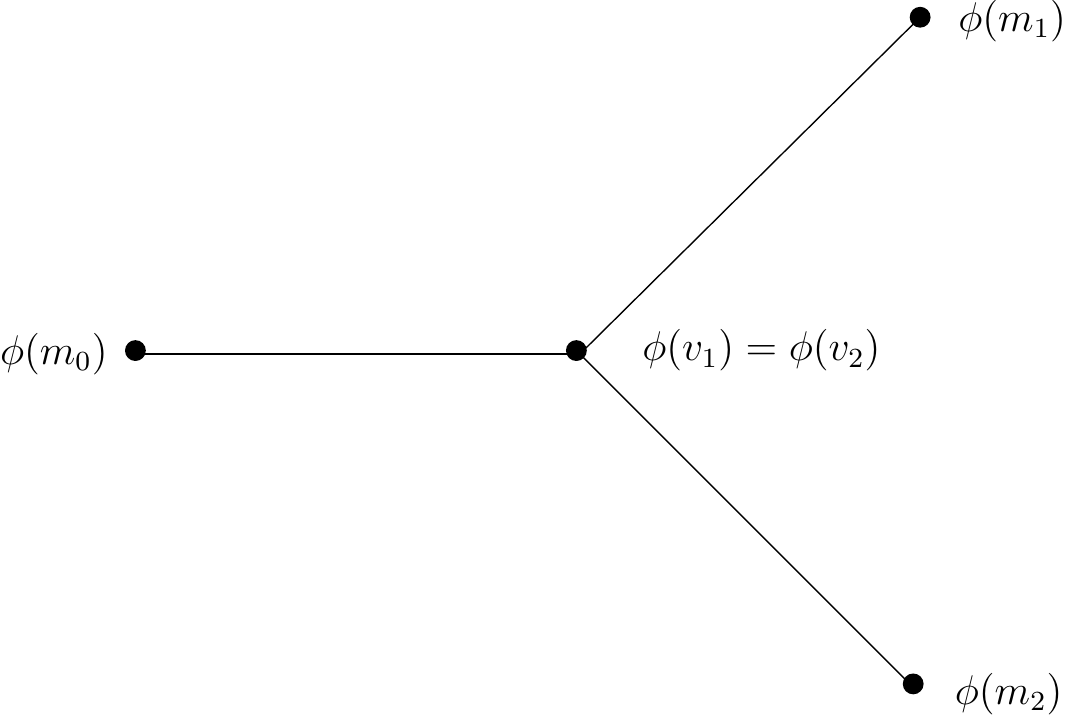}
\caption{In case $g(\phi (G_0)=0$}\label{Figuur 3}
\end{center}
\end{figure}

Consider the loop $e_1=v_1m_1v_2m_0v_1$.
Since $\phi (G_0)$ has genus 0 it follows $\phi (e_1)$ is a subtree $T_1$ of $\phi (G_0)$.
Since $T_1$ is the image of a loop and $\deg (\phi)=2$, it follows that $\sharp ( (\phi \bigr| _{e_1})^{-1}(q))=1$ for some $q \in \phi (e_1)$ if and only if $q$ is a point of valence 1 of $\phi (e_1)$.
Since $\deg (\phi) = 2$ it follows  $d_q(\phi)=2$ for such point $q$ and in particular $d_q(\phi)=1$ for $q \in e_1$ if $\phi (q)$ does not have valence 1 on $\phi (e_1)$.
In case $\phi (e_1)$ would have a point $q'$ of valence 3 then there exist at least 3 different points $q$ on $e_1$ with $\phi (q)=q'$, contradicting $\deg (\phi)=2$.
Hence $\phi (e_1)$ can be considered as a finite edge with two vertices.
Also for each $q\in e_1$ and $v \in T_q(e_1)$ one has $d_v(\phi ) =1$.

Assume $\phi (v_1) \neq \phi (v_2)$.
Then $\phi (v_1m_2v_2)$ is a path from $\phi (v_1)$ to $\phi (v_2)$ outside of $\phi (e_1)$.
This would imply $g(\phi (G_0))\geq 1$, contradicting $g(\phi (G_0))=0$, hence $\phi (v_1) = \phi (v_2)$.
Repeating the previous arguments for the loop $e_2=v_1m_2v_2m_0v_1$ one obtains the given description for $\phi \bigr| {G_0} : G_0 \rightarrow \phi (G_0)$.

Consider $\phi (q_1) \in \phi (v_1m_1v_2)$ and $q'_1 \neq q_1$ on $v_1m_1v_2$ with $\phi (q_1)=\phi (q'_1)$.
We obtain $d_{q_1}(\phi)=1$.
In case $\phi (\gamma_1)$ would be a tree then it would imply $\phi (q_1)$ cannot be a point of valence 1 on $\phi (\gamma_1)$ hence there exists $q'' _1 \in \gamma_1 \setminus \{ q_1 \}$ with $\phi (q_1)=\phi (q''_1)$.
Hence $\sharp ( \phi^{-1}(\phi (q)))\geq 3$, contradicting $\deg (\phi)=2$.
It follows $g(\phi(\gamma_1))=1$, hence $\phi (\gamma_1)$ contains a loop $e'_1$ in $E$.
Using the same argument as before we obtain $\phi (q_1) \in e'_1$.
Also $e'_1 \cap \phi (G_0) = \{ \phi (q_1) \}$.

Repeating the arguments using $q_2$ and $\gamma_2$ we obtain a loop $e'_2$ in $E$ such that $e'_2 \cap \phi (G_0)= \{ \phi (q_2) \}$, hence $e'_1 \cap e'_2 = \emptyset$.
Since $g(E)=1$ this is impossible.
As a conclusion we obtain $g(\phi(G_0))=1$.

In case $\phi (e_1)$ would have genus 0, from the previous arguments it follows that for $q \in v_1m_2v_2 \setminus \{ v_1, v_2 \}$ one has $\phi (q) \notin \phi (e_1)$.
In case $\phi (v_1) \neq \phi (v_2)$ it implies $\phi(e_2)$ has genus 1 with $e_2=v_1m_0v_2m_1v_1$.
In case $\phi (v_1) = \phi (v_2)$ and $g(\phi (e_2))=0$ too, it would imply $g(\phi (G_0))=0$.
Hence we can assume $\phi (e_1)$ has genus 1 (but then $\phi (e_2)$ could have genus 0).

\noindent \underline{Step 2:} $\phi \bigr | _{e_1} : e_1 \rightarrow \phi (e_1)$ is an isomorphism (meaning it is finite harmonic of degree 1)

Since $g(\phi (e_1))=1$ it follows there is a loop $e$ in $\phi (e_1)$, finitely many points $r_1, \cdots , r_t$ on $e$ and finitely many trees $T_i$ inside $\phi (e_1)$ such that

\noindent \hspace{1cm} $T_i \cap e = \{ r_i \}$

\noindent \hspace{1cm} $T_i \cap T_j = \emptyset$ in case $i \neq j$

\noindent \hspace{1cm} $\phi (e_1)= e \cup T_1 \cup \cdots \cup T_t$

\noindent (of course $t=0$, hence $\phi (e_1)$ is a loop, is also possible; we are going to prove that $t=0$).

\begin{figure}[h]
\begin{center}
\includegraphics[height=3 cm]{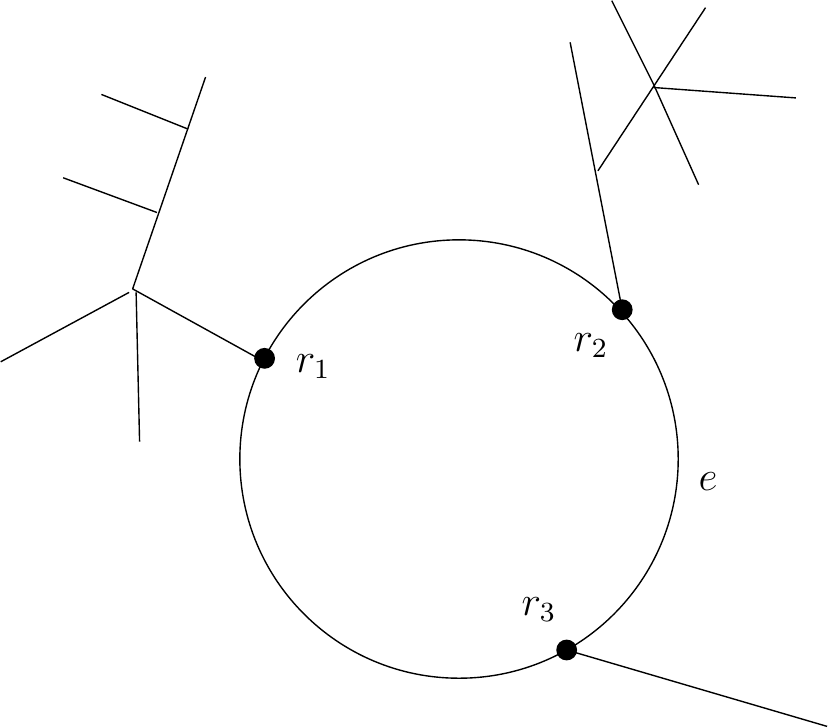}
\caption{$g(\phi (e_1))=1$}\label{Figuur 4}
\end{center}
\end{figure}
In case $\val _{\phi (e_1)}(r_i) > 3$ for some $1 \leq i \leq t$ then $\phi ^{-1} (r_i)$ contain at least 3 different points on $e_1$, contradicting $\deg (\phi)=2$.
So we obtain a situation like in Figure \ref{Figuur 4}. Let $r\in \{r_1, \cdots ,r_t\}$ and let $T$ be the associated subtree of $\phi (e_1)$.
Then $\phi ^{-1} (r) = \{ r', r'' \} \subset e_1$ with $r' \neq r''$.
The tangent space $T_r (\phi (e_1))$ consists of 3 elements (see Figure \ref{Figuur 5}).
Hence there exists $w \in T_{r'}(G_0) \setminus T_{r'}(e_1)$ such that $d_{\phi}(r' ) (w) \in T_r(e)$.
Let $f$ be the edge of $G'_n$ defining $w$ hence $\phi (f) \subset e$.
Because of Lemma \ref{lemma3} this implies $r'$ is one of the points $q_i$ on $G_0$ and $f \subset \gamma _i$.
Repeating the same argument using $r''$ instead of $r'$ one obtains a contradiction to $\deg (\phi)=2$.
This proves $t=0$, hence $\phi (e_1)=e$ is a loop.
\begin{figure}[h]
\begin{center}
\includegraphics[height=3 cm]{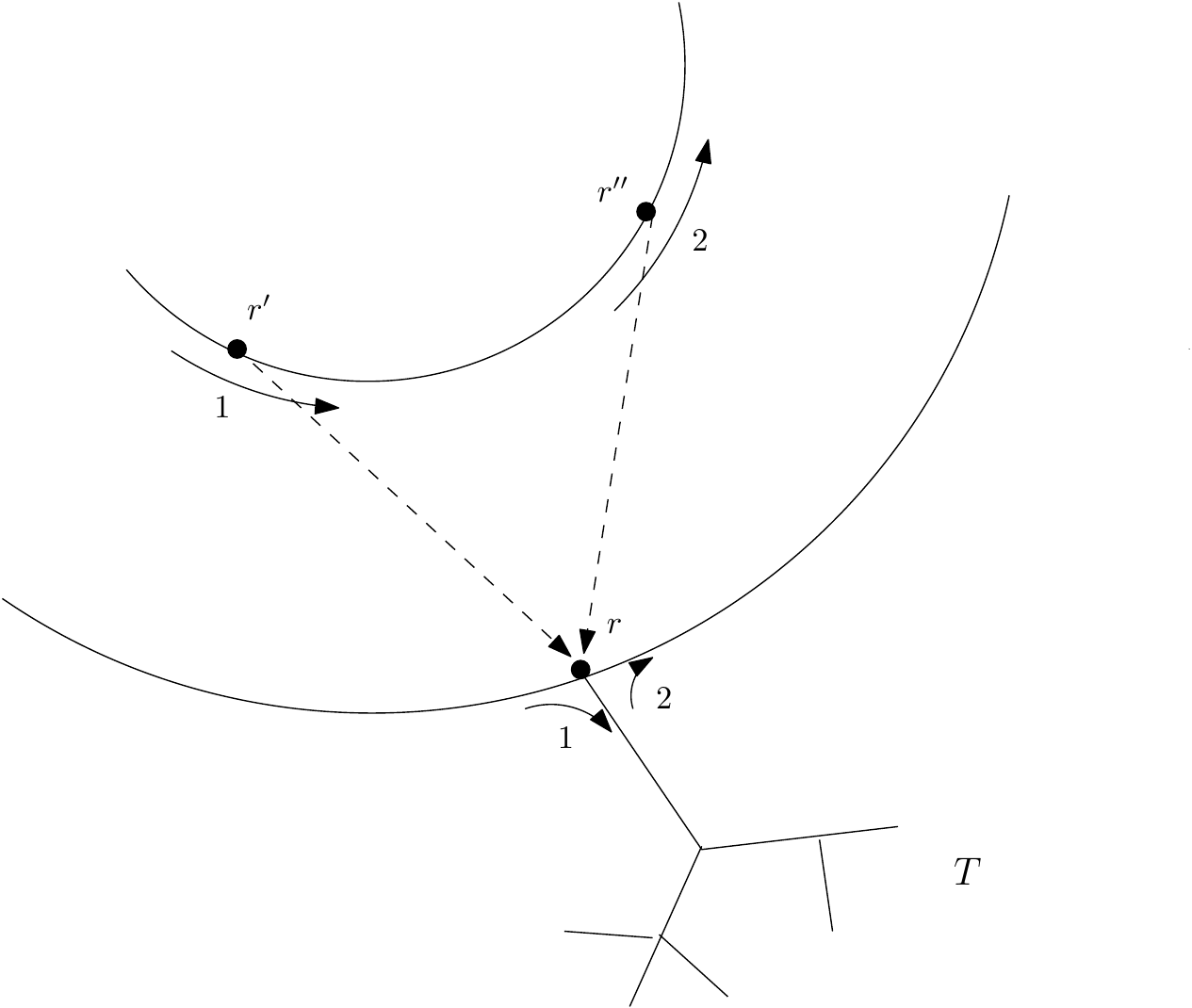}
\caption{In case $t\neq 0$}\label{Figuur 5}
\end{center}
\end{figure}

Because $\deg (\phi)=2$ we cannot go back and forth on $e$ following $e_1$ and taking the image under $\phi$.
In principle it could be the case that there exist different points $q', q''$ on $e_1$ such that the image of the closure of both components of $e_1 \setminus \{ q', q'' \}$ is equal to $e$ with $d_{q'}(\phi)=d_{q''}(\phi)=2$.
This would correspond to something like shown in Figure \ref{Figuur 6}.
\begin{figure}[h]
\begin{center}
\includegraphics[height=5 cm]{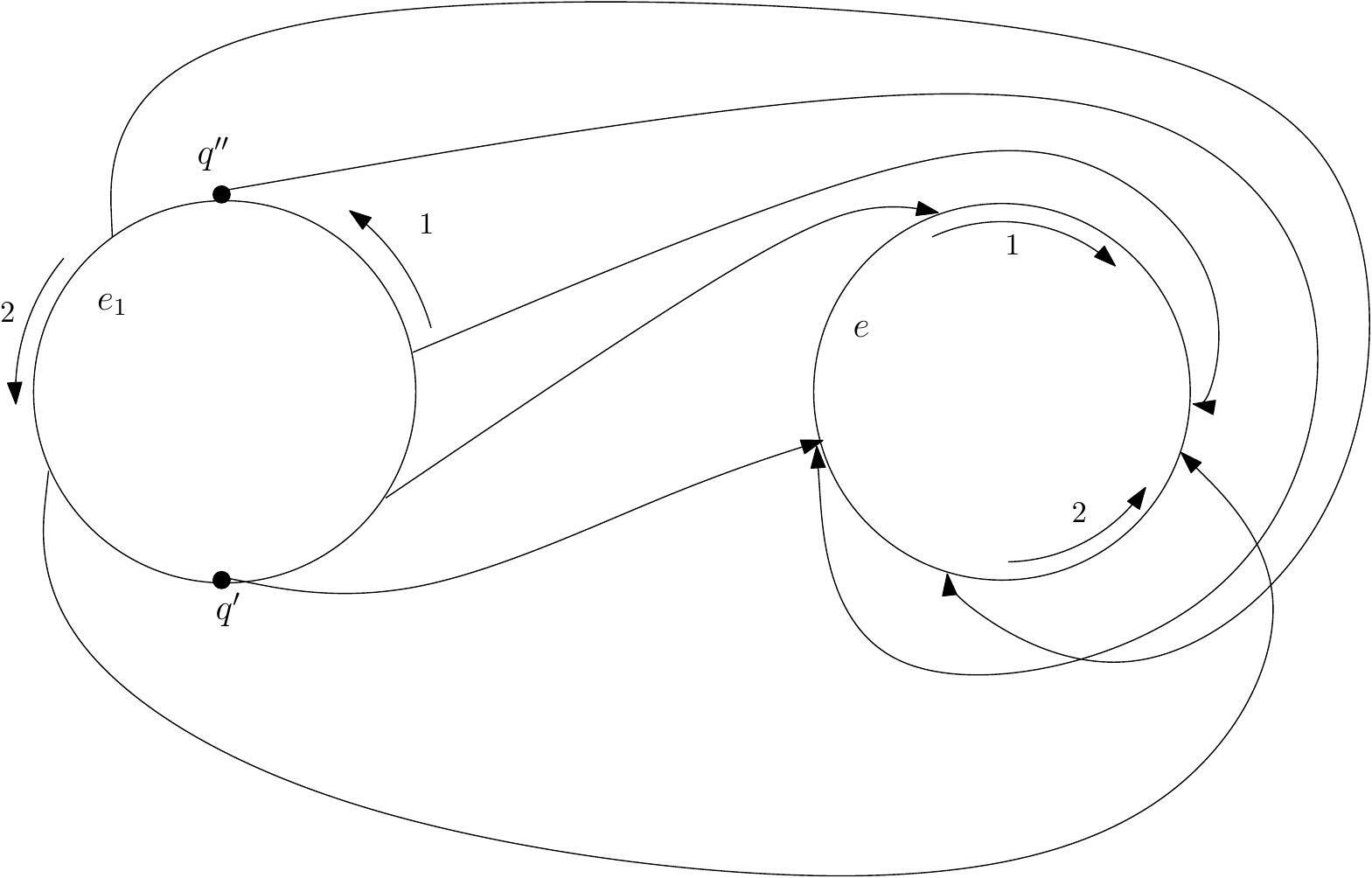}
\caption{A case that cannot occur}\label{Figuur 6}
\end{center}
\end{figure}
In that case there should exist $v\in T_{q'}(G'_n)$ with $d_{\phi}(q')(v) \in T_{\phi (q')}(e) \setminus d_{\phi}(q')(T_{q'}(e_1))$.
Hence the edge $f$ of $G'_n$ defining $v$ satisfies $\phi(f) \subset e$, contradicting $\deg (\phi)=2$.
Hence the situation from Figure \ref{Figuur 6} cannot occur.
It follows that in case there exist $q' \neq q''$ on $e_1$ such that $\phi (q')=\phi (q'')$ then $\phi \bigr | _{e_1} : e_1 \rightarrow e$ is harmonic of degree 2 and $d_q(\phi \bigr | _{e_1})=1$ for all $q \in e_1$.
In this case $\phi (v_1m_2v_2) \cap e = \{ \phi (v_1), \phi (v_2) \}$ since $\deg (\phi )=2$.
In case $\phi (v_1) \neq \phi (v_2)$ then this contradicts $g(E)=1$.
In case $\phi (v_1) = \phi (v_2)$ then because of the description of $\phi \bigr |_{e_1}$ one has $l(v_1m_1v_2)=l(v_1m_0v_2)$.
We assume this is not the case, so we can assume $\phi \bigr |_{e_1} : e_1 \rightarrow e$ is bijective.

In case for each edge $f$ on $G'_n$ with $f \subset e_1$ one has $d_f(\phi)=2$ then again, since $\phi (v_1) \neq \phi (v_2)$ we have $\phi (v_1m_2v_2) \cap e = \{ \phi (v_1), \phi (v_2) \}$, contradicting $g(E)=1$.
Assume there exists $q\in e_1$ being a vertex of $G'_n$ and two edges $e', e''$ of $G'_n$ contained in $e_1$ with vertex end point $q$ such that $d_{e'}(\phi)=1$ and $d_{e''}(\phi)=2$.
In particular it follows $d_q(\phi)=2$.
Let $v' \in T_q(G'_n)$ correspond to $e'$ then there exists $v\in T_q(G'_n)$ with $v \notin T_q(e_1)$ such that $d_{\phi}(q)(v)=d_{\phi}(q)(v')$.
Let $f$ be the edge of $G'_n$ defining $v$, then $\phi (f) \subset e$.
From Lemma \ref{lemma3} it follows $q$ is one of the points $q_i$ and $f \subset \gamma_i$.
Since $g(E)=1$ it follows $e \subset \phi (\gamma)$, but this is impossible because $d_{e''}(\phi)=2$ and $\deg (\phi)=2$.
This proves $\phi \bigr |_{e_1} : e_1 \rightarrow e$ is an isomorphism of metric graphs.

\noindent \underline{Step 3:} Finishing the proof of the theorem.

\noindent It follows $\phi (v_1)$ and $\phi (v_2)$ do split $e$ into two parts $e'$ and $e''$ of lengths $l(v_1m_1v_2)$ and $l(v_1m_0v_2$.
Since $g(E)=1$ it follows $\phi (v_1m_2v_2)$ contains $e'$ or $e''$, we assume it contains $e'$.
In case $\phi (v_1m_2v_2)$ would contain $\tilde {q}\in e'' \setminus \{ \phi (v_1), \phi (v_2) \}$ then because of $g(E)=1$ it follows $\phi (v_1m_2v_2)$ contains one of the connected components of $e'' \setminus \{ \tilde {q} \}$.
On that connected component we get a contradiction to $\deg (\phi)=2$.

So  we obtain different points $r_1 , \cdots , r_t$ on $e'$ and trees $T_1 , \cdots , T_t$ with $T_i \cap e'=r_i$ for $1 \leq i \leq t$ and $T_i \cap T_j = \emptyset$ for $i \neq j$ such that $\phi (v_1m_2v_2)= e' \cup T_1 \cup \cdots \cup T_t$.
It is possible (and we are going to prove) that $t=0$, hence $e'=\phi (v_1m_2v_2)$.
In case $r_i \notin \{ \phi (v_1), \phi (v_2) \}$ then there exist two different points $r', r''$ on $v_1m_2v_2$ such that $\phi (r') = \phi (r'') = r_i$.
Since $r_i$ is alse the image of a point on $e_1$ we get a contradiction to $\deg (\phi)=2$.
Hence $t \leq 2$ and $r_i \in \{ \phi (v_1), \phi (v2) \}$.

Assume $r_i = \phi (v_1)$.
We obtain $q \in v_1m_2v_2$ with $q \notin \{ v_1, v_2 \}$ and $\phi (q) = \phi (v_1)$.
There exists $v \in T_q(G'_n)$ such that $d_{\phi}(q)(v)$ is the element of $T_{\phi (q)}(E)$ defined by $e''$.
Let $f$ be the edge of $G'_n$ defining $v$.
From Lemma \ref{lemma3} it follows $q$ is one of the points $q_i$ and $f \subset \gamma_i$.
Since $g(E)=1$ we obtain $\phi (\gamma_i)$ contains $e$.
This implies that for $P \in e'$ there are at least 3 points contained in $\phi ^{-1}(P)$, a contradiction.
This proves $t=0$, hence $\phi (v_1m_2v_2)=e'$.
Since $\deg (\phi)=2$ it also implies $d_f(\phi)=1$ for each edge $f$ contained in $v_1m_2v_2$, hence $l(e')=l(v_1m_2v_2)$.
Since $l(v_1m_2v_2) \notin \{ l(v_1m_0v_2), l(v_1m_1v_2) \}$ we obtain a contradiction, finishing the proof of the theorem.
\end{proof}

As a corollary of the theorem we obtain the goal of this paper.

\begin{corollary}
For each genus $g \geq 5$ there is metric graph $\Gamma$ of genus $g$ satisfying $w^1_4 \geq 1$ that has no divisor of Clifford index at most 1 and is not tropically equivalent to a metric graph $\Gamma'$ such that there exists a finite harmonic morphism $\pi : \Gamma' \rightarrow E$ of degree 2 with $g(E)=1$. In particular in case $g \geq 6$ the graph $\Gamma$ cannot be lifted to a curve $X$ of genus $g$ satisfying $\dim (W^1_4)=1$.
\end{corollary}

\end{document}